\documentclass[12pt,reqno]{article}
mm \textheight=8.9in

\usepackage{amssymb}
\usepackage{amsmath}
\usepackage{amsthm}

\usepackage{color}

\newcommand{\br}[3]{{$#1$}$\lower4pt\hbox{$\tp\atop\raise4pt \hbox{$\scriptscriptstyle{#2}$}$} ${$#3$}}
\newcommand{\tw}[3]{{$#1$}${\,\scriptscriptstyle {#2}}\atop\raise9pt\hbox{$\scriptstyle\tp$} ${$#3$}}
\newcommand{\ttps}[2]{{#1}\raise5pt\hbox{$\lower12pt\hbox{$\scriptstyle\tp$}\atop \lower0pt\hbox{$\tilde\;$}$}\raise4.5pt\hbox{${\scriptstyle{#2}}$}}
\newcommand{\st}[1]{\mbox{${\,\scriptscriptstyle {#1}}\atop\raise5.5pt\hbox{$*$}$}}

\newcommand{\rd}[1]{\mbox{${\,\scriptscriptstyle {#1}}\atop\raise5.5pt\hbox{$\bullet$}$}}
\newcommand{\rt}[1]{\otimes_\chi}
\newcommand{\lt}[1]{\mbox{${\,\scriptscriptstyle {#1}}\atop\raise5.5pt\hbox{$\ltimes$}$}}
\newcommand{\btr}{\raise1.2pt\hbox{$\scriptstyle\blacktriangleright$}\hspace{2pt}}
\newcommand{\btl}{\raise1.2pt\hbox{$\scriptstyle\blacktriangleleft$}\hspace{2pt}}

\newcommand{\lcr}{\raise1.0pt \hbox{${\scriptstyle\rightharpoonup}$}}
\newcommand{\rcr}{\raise1.0pt \hbox{${\scriptstyle\leftharpoonup}$}}

\newcommand{\ttp}{{\lower12pt\hbox{$\tp$}\atop \hbox{$\tilde\;$}}}

\newcommand{\id}{\mathrm{id}}

\newcommand{\Hc}{\mathcal{H}}

\newcommand{\Hg}{\mathfrak{H}}

\newcommand{\Sc}{\mathcal{S}}
\renewcommand{\S}{\mathcal{S}}
\newcommand{\Ru}{\mathcal{R}}

\newcommand{\Cc}{\mathcal{C}}

\renewcommand{\O}{\mathcal{O}}

\newcommand{\C}{\mathbb{C}}

\newcommand{\Qbb}{\mathbb{Q}}

\newcommand{\Z}{\mathbb{Z}}

\newcommand{\N}{\mathbb{N}}

\newcommand{\tp}{\otimes}

\newcommand{\zt}{\zeta}

\newcommand{\U}{U}

\newcommand{\gm}{\gamma}
\newcommand{\dt}{\delta}

\newcommand{\op}{\oplus}
\newcommand{\la}{\lambda}

\newcommand{\End}{\mathrm{End}}

\newcommand{\rk}{\mathrm{rk}}

\newcommand{\Rm}{\mathrm{R}}

\newcommand{\ad}{\mathrm{ad}}

\newcommand{\Ga}{\Gamma}
\newcommand{\g}{\mathfrak{g}}
\newcommand{\e}{\mathfrak{e}}
\renewcommand{\b}{\mathfrak{b}}

\newcommand{\h}{\mathfrak{h}}

\newcommand{\s}{\mathfrak{s}}

\newcommand{\f}{\mathfrak{f}}

\newcommand{\eps}{\epsilon}

\newcommand{\p}{\mathfrak{p}}
\renewcommand{\l}{\mathfrak{l}}

\newcommand{\si}{\sigma}
\newcommand{\al}{\alpha}

\renewcommand{\b}{\mathfrak{b}}
\newcommand{\bt}{\beta}

\newcommand{\be}{\begin{eqnarray}}
\newcommand{\ee}{\end{eqnarray}}

\newtheorem{thm}{Theorem}[section]
\newtheorem{propn}[thm]{Proposition}
\newtheorem{lemma}[thm]{Lemma}

\newtheorem{conjecture}{Conjecture}

\newtheorem{definition}[thm]{Definition}

\newcount\prg

\newcommand{\parag}{\advance\prg by1 {\noindent\bf\thesection.\the\prg\hspace{6pt}}}

\begin{document}
\title{ Shapovalov elements of classical and quantum groups}
\author{
Andrey Mudrov \vspace{10pt}\\
\vspace{10pt}
\small
\sl In memorium of Vladimir Lyachovsky
\\
\small
 Moscow Institute of Physics and Technology,\\
\small
9 Institutskiy per., Dolgoprudny, Moscow Region,
141701, Russia,
\vspace{10pt}\\
\small University of Leicester, \\
\small University Road,
LE1 7RH Leicester, UK,
\vspace{10pt}\\
\small e-mail:  am405@le.ac.uk
\\
}

\date{ }

\maketitle

\begin{abstract}
Shapovalov elements $\theta _{\beta,m}$
of the classical or quantized universal enveloping algebra of
  a simple Lie algebra $\g$ are
parameterized by a positive root $\beta$ and a positive integer $m$.
They relate the
highest vector of
a reducible Verma module with highest vectors of its submodules.
We obtain a factorization of $\theta_{\beta,m}$ to  a product of $\theta_{\beta,1}$
and calculate $\theta_{\beta,1}$ as a residue of  a matrix element of the inverse Shapovalov form
via a generalized Nigel-Moshinsky algorithm. This way  we  explicitly express
$\theta_{\beta,m}$ of  a  classical simple Lie algebra through the Cartan-Weyl basis in $\mathfrak{g}$. In the case of quantum groups,
we give an analogous formulation through  the entries of the R-matrix (quantum $L$-operator) in fundamental representations.
\end{abstract}

\\
{\small \underline{AMS classification codes}: 17B10, 17B37}
 \newpage
 \section{Introduction}
Category $\O$ introduced  in  \cite{BGG1} for
semi-simple Lie algebras and  later defined for many other
classes of algebras   including quantum groups plays a fundamental role  in various fields of mathematics and mathematical physics.
In particular, it accommodates finite-dimensional and numerous important infinite dimensional representations like  parabolic Verma modules and
their  generalizations \cite{H}. There are distinguished objects in $\O$ called Verma modules that feature a universality
property: all simple modules in $\O$ are their quotients.   The maximal proper submodule
in a Verma module is  generated by extremal vectors \cite{BGG2}, which are invariants
of the positive triangular  subalgebra.
This makes extremal vectors critically important in representation theory.

Extremal vectors in a Verma module $V_\la$ are related with the vacuum vector of highest weight $\la$  via special elements
$\theta_{\bt,m}$ of the (classical or quantum) universal enveloping of the  negative Borel subalgebra that are
called Shapovalov elements \cite{Shap,C}.
They are parameterized with a positive root $\bt$ and an integer $m\in \N$
validating a  De Concini-Kac-Kazhdan condition on   $\la$. In the classical version, it is $2(\la+\rho,\bt)-m(\bt,\bt)=0$  with $\rho$
being the half-sum of positive roots. This  condition guarantees that the Verma module is reducible. In the special case when the root $\bt$ is
simple, the element $\theta_{\bt,m}$ is a power $f^m_\bt$ of the simple root vector $f_\bt$ of weight $\bt$.  For non-simple $\bt$,
the Shapovalov elements are  complicated polynomials in negative Chevalley generators with coefficients in the Cartan subalgebra.
It can be viewed as a function $\theta_{\bt,m}(\la)$ of the highest weight  of a generic Verma module $V_\la$
with values in the subalgebra generated by negative root vectors.

A description of extremal vectors in Verma modules over classical Kac-Moody algebras was obtained in   \cite{MFF}
via an interpolation procedure resulting in a  calculus of polynomials with complex exponents.
Another approach based on extremal projectors \cite{AST}
 was employed by Zhelobenko in \cite{Zh}. He obtained $\theta_{\bt,m}$ for  simple Lie algebras
as a product of $m$ copies of $\theta_{\bt,1}$ with shifted weights. The idea of factorization was also used in a construction of Shapovalov elements
for contragredient Lie superalgebras  in \cite{Mus}.

Factorization of $\theta_{\bt,m}$ into a product of polynomials of lower degree is a great simplification that is convenient
for their  analysis. For example it is
good for the study of the classical limit in the case of quantum groups, which is
crucial for quantization of conjugacy classes   \cite{M2}.

With regard to quantum groups, an inductive construction of extremal vectors in Verma modules
 was suggested in  \cite{KL}. Explicit expressions for  Shapovalov elements   for the $A$-type appeared  in \cite{M3} and recently were obtained in \cite{CM} by other methods.
It is worthy to note that ordered PBW-like monomials in $\theta_{\bt,1}$ deliver an orthogonal basis in generic irreducible Verma modules \cite{M3}.

While  Zhelobenko's factorization via extremal projectors simplifies construction of Shapovalov elements in the case of  classical simple Lie algebras,
 there remains a problem of explicit description of the factors.
In this paper, we pursue  an alternative approach   based on the canonical contravariant  bilinear form on Verma modules.
It  gives  expressions for all Shapovalov elements in a factorized form through root vectors in
the classical case and through elements of the R-matrix in the adjoint representation in the case of quantum groups.

Extremal vectors generate the kernel of the canonical contravariant form on $V_\la$,
which is a specialization at $\la$ of the "universal" Shapovalov form on the Borel subalgebra \cite{Shap} with values in the Cartan subalgebra.
This contravariant form itself is extremely important and has numerous applications, see e.g. \cite{ES,EK,FTV,AL}.
For generic $\la$, the  module $V_\la$ is irreducible and  the form
is non-degenerate. The inverse form gives rise to an element $\Sc$
 of extended tensor product of positive and negative subalgebras of the (quantized) universal
 enveloping algebra \cite{M1}.
 Sending the positive leg of $\Sc$ to an auxiliary representation yields a matrix with entries in the negative
 subalgebra which we call   Shapovalov matrix.
Its explicit description    was obtained in \cite{M1} by generalization of Nagel-Moshinski
formulas for the lowering operators of $\s\l(n)$ \cite{NM}. They can also be
derived (in the quantum setting) from the ABRR equation \cite{ABRR} on  dynamical twist \cite{ES}.


Our method relates $\theta_{\bt,m}$ with certain entries
of the   Shapovalov matrix. This point of view is quite natural because  the kernel of the contravariant form on $V_\la$
results in poles of $\S$.
Our approach  not only provides a factorization of $\theta_{\bt,m}$
to a product of $\theta_{\bt,1}$ but also an efficient  description
of $\theta_{\bt,1}$ in a very elementary way, by a generalized Nagel-Moshinsky rule (\ref{norm_sing_vec})
using a technique of Hasse diagrams.
We do it by evaluating residues of matrix elements of $\Sc$
that go singular at a De Concini-Kac-Kazhdan "hyperplane".

Our approach is absolutely parallel for a classical semi-simple Lie algebra $\g$
and its quantum group $U_q(\g)$. The classical case can be processed directly or obtained as the limit case $q\to 1$
of the deformation parameter. Let us describe the  method in more detail.

With a module $V$ from the category $\O$ and a pair of non-zero
vectors $v_b,v_a\in V$
we associate a Shapovalov matrix element, $\langle v_b|v_a\rangle$, which belongs to the negative Borel (universal enveloping) subalgebra
$\hat  U_q(\b_-)$ rationally extended over
 the Cartan subalgebra.
Under   certain assumptions  on $v_b$ and $v_a$, such  matrix elements deliver factors in $\theta_{\bt,m}$.
These factors normalize positive root vectors of a reductive Lie subalgebra  $\l\subset \g$
whose negative counterparts annihilate $v_b$. This way they become lowering operators in the Mickelsson algebras of the pair $(\g,\l)$,
\cite{Mick}. When $\la$ satisfies  the De Concini-Kac-Kazhdan condition, the factors become  $\theta_{\bt,1}$ shifted by certain weights.

The vector  $v_b$ should be highest for the minimal simple Lie subalgebra in $\g$ that accommodates the root $\bt$ and
and its weight should satisfy the condition $(\nu_b,\bt)=\frac{(\bt,\bt)}{2}$. For finite dimensional $V$ the latter is equivalent to saying
that $v_b$ generates a  $2$-dimensional submodule
of the $\s\l(2)$-subalgebra generated
by the root spaces $\g_{\pm \bt}$.

The vector  $v_a$ determines a
homomorphism $V_{\la_2}\to V\tp V_{\la_1}$, where $V_{\la_i}$ are irreducible Verma modules of
highest weights $\la_i$ and $\la_2-\la_1$ equals  the weight of $v_a$.
Iteration of this construction yields a chain of homomorphisms
$$
V_{\la_{m}}\to V\tp V_{\la_{m-1}}\to \ldots \to  V^{\tp m}\tp V_{\la_0},
$$
where each mapping $V^{\tp i}\tp V_{\la_{m-i}}\to V^{\tp i}\tp (V\tp V_{\la_{m-i-1}})$ is identical on the factor $V^{\tp i}$.
We prove factorization of  $\langle v_b^{\tp m}| v_a^{\tp m}\rangle$ to
a product of $\langle v_b|v_a\rangle$. Then we demonstrate  that, under the specified conditions, that matrix element is proportional to  $\theta_{\bt,m}(\la)$
 with $\la_0=\la$.

As a result, we obtain $\theta_{\bt,m}(\la)$ as a product  $\prod_{i=0}^{m-1}\theta_{\bt,1}(\la_i)$.
The factors $\theta_{\bt,1}$ are calculated   by a
general rule (\ref{norm_sing_vec}) specialized to the case in Section 5.
Viewed as an element of $\hat  U_q(\b_-)$,
$\theta_{\bt,m}$ becomes a product of  $\theta_{\bt,1}$ shifted by the integer multiple weights of $v_b$.
This shift can be made trivial by a choice of $V$
if $\bt$ contains a
simple root $\al$   with multiplicity $1$.
Then $\theta_{\bt,m}$ becomes the $m$-th power of $\theta_{\bt,1}$.

It is worthwhile mentioning that $\theta_{\bt,1}$ can be obtained via an arbitrary auxiliary module $V$ with a pair of vectors $(v_a,v_b=e_\bt v_a)$.
They all coincide up to a scalar factor on the De Concini-Kac-Kazhdan "hyperplane" and generally differ away from it.
The problem is to use $\theta_{\bt,1}$ as a factor block for constructing $\theta_{\bt, m}$ of higher $m$.
That is why we choose $(V,v_b,v_a)$ in a special way as described above. On the other hand,
since the left tensor leg of $\Sc$ is in the positive subalgebra $U_q(\g_+)\subset U_q(\g)$,
it is the structure of  $U_q(\g_+)$-submodule on $V$ that determines $\theta_{\bt,1}$.
A  remarkable fact is that the cyclic submodule $U_q(\g_+)v_a$ in an admissible $V$  turns out to be isomorphic to
a subquotient of the $U_q(\g_+)$-module corresponding to $\g/\g_+$ in the classical limit.
This means that $\theta_{\bt,1}$ in each case can be calculated via Shapovalov matrix elements
from  $\End(\g/\g_+)\tp U_q(\b_-)$, by Theorem \ref{classifying Hasse}.

Except for Section 5, we present only the $q$-version of the theory. The classical case can be obtained by sending $q$ to $1$.
However the final expression for $\theta_{\bt,1}$ is  greatly simplified when $q=1$, so we give a special consideration to it in
the  Section 5.

\section{Preliminaries}

\label{SecPrelim}

Let  $\g$ be a  simple complex Lie algebra and  $\h\subset \g$ its Cartan subalgebra. Fix
a triangular decomposition  $\g=\g_-\op \h\op \g_+$  with  maximal nilpotent Lie subalgebras
$\g_\pm$.
Denote by  $\Rm \subset \h^*$ the root system of $\g$, and by $\Rm^+$ the subset of positive roots with basis $\Pi$
of simple roots. This basis  generates a root lattice $\Ga\subset \h^*$ with the positive semigroup $\Ga_+=\Z_+\Pi\subset \Ga$.

For a positive root $\bt\in \Rm^+$ and a simple root $\al\in \Pi$ denote by $\ell_{\al,\bt}\in \Z_+$ the multiplicity
with which $\al$ enters $\bt$, that is the $\al$-coefficient in the expansion of $\bt$ over the basis $\Pi$.

Choose an $\ad$-invariant form $(\>.\>,\>.\>)$ on $\g$, restrict it to $\h$, and transfer to $\h^*$ by duality.
For every $\la\in \h^*$ there is   a unique element $h_\la \in \h$ such that $\mu(h_\la)=(\mu,\la)$, for all $\mu\in \h^*$.
For a non-isotropic $\mu\in \h^*$  set  $\mu^\vee=\frac{2}{(\mu,\mu)}\mu$ and  $h_\mu^\vee=\frac{2}{(\mu,\mu)}h_\mu$.

Fundamental weights are denoted by $\omega_\al$, $\al \in \Pi$. They  are determined by the system of equations
$(\omega_\al, \bt^\vee)=\dt_{\al,\bt}$, for all $\al, \bt \in \Pi$.

We assume that $q\in \C$ is not a root of unity and we understand that when saying "all $q$". By almost all $q$ we mean  all $q$ excepting maybe  a finite set of values distinct
from $q=1$.

The standard Drinfeld-Jimbo quantum group $U_q(\g)$   is a complex Hopf algebra with the set of generators $e_\al$, $f_\al$, and $q^{\pm h_\al}$ labeled by simple roots $\al$ and satisfying relations \cite{D1,J}
$$
q^{h_\al}e_\bt=q^{ (\al,\bt)}e_\bt q^{ h_\al},
\quad
[e_\al,f_\bt]=\dt_{\al,\bt}[h_\al]_q,
\quad
q^{ h_\al}f_\bt=q^{-(\al,\bt)}f_\bt q^{ h_\al},\quad \al, \bt \in \Pi.
$$
The symbol $[z]_q$, where  $z\in \h+\C$, stands for $\frac{q^{z}-q^{-z }}{q-q^{-1}}$.
The elements $q^{h_\al}$ are invertible, with $q^{h_\al}q^{-h_\al}=1$, while  $\{e_\al\}_{\al\in \Pi}$ and $\{f_\al\}_{\al\in \Pi}$ also
satisfy quantized Serre relations. Their exact form is not important for this presentation, see \cite{ChP} for details.

A Hopf algebra structure on $U_q(\g)$ is introduced by the comultiplication
$$
\Delta(f_\al)= f_\al\tp 1+q^{-h_\al}\tp f_\al,\quad \Delta(q^{\pm h_\al})=q^{\pm h_\al}\tp q^{\pm h_\al},\quad\Delta(e_\al)= e_\al\tp q^{h_\al}+1\tp e_\al
$$
set up on the generators and extended as a homomorphism $U_q(\g)\to U_q(\g)\tp U_q(\g)$.
The antipode is an algebra and coalgebra anti-automorphism of $U_q(\g)$ that acts on the generators by the assignment
$$
\gm( f_\al)=- q^{h_\al}f_\al, \quad \gm( q^{\pm h_\al})=q^{\mp h_\al}, \quad \gm( e_\al)=- e_\al q^{-h_\al}.
$$
The counit homomorphism $\eps\colon U_q(\g)\to \C$ returns
$$
\eps(e_\al)=0, \quad \eps(f_\al)=0, \quad \eps(q^{h_\al})=1.
$$
We extend the notation $f_\al$, $e_\al$ to all  $\al\in \Rm^+$ meaning the Lusztig root vectors
with respect to some normal ordering of $\Rm^+$, \cite{ChP}. They are known to generate a Poincare-Birkhoff-Witt (PBW)
basis in $U_q(\g_\pm)$.

Denote by $U_q(\h)$,  $U_q(\g_+)$, and $U_q(\g_-)$  subalgebras   in $U_q(\g)$
 generated by $\{q^{\pm h_\al}\}_{\al\in \Pi}$, $\{e_\al\}_{\al\in \Pi}$, and $\{f_\al\}_{\al\in \Pi}$, respectively.
The quantum Borel subgroups are defined as $U_q(\b_\pm)=U_q(\g_\pm)U_q(\h)$; they are Hopf subalgebras in $U_q(\g)$.
We will also need their extended version $\hat U_q(\b_\pm)=U_q(\g_\pm)\hat U_q(\h)$, where
$\hat U_q(\h)$ is the ring of fractions of $U_q(\h)$ over the multiplicative system generated by
$[h_\al-c]_q$ with $\al \in \Gamma_+$ and $c\in \Qbb$.

Given a $U_q(\g)$-module $V$, a non-zero vector $v$ is said to be of weight $\mu$ if $q^{h_\al}v=q^{(\mu,\al)} v$ for all $\al\in \Pi$.
The linear span of such vectors is denoted by $V[\mu]$.
A module $V$ is said to be of highest weight $\la$ if it is generated by a weight vector $v\in V[\la]$ that
is killed by all $e_\al$. Such vector $v$ is called highest; it is defined up to a non-zero scalar multiplier.

We define an involutive coalgebra anti-automorphism and algebra automorphism $\si$ of $U_q(\g)$ setting it on the
generators by the assignment
$$\si\colon e_\al\mapsto f_\al, \quad\si\colon f_\al\mapsto e_\al, \quad \si\colon q^{h_\al}\mapsto q^{-h_\al}.$$
The involution $\omega =\gamma^{-1}\circ \si=\si \circ \gm$ is an algebra anti-automorphism of $U_q(\g)$ and preserves
the comultiplication.

A symmetric bilinear form $(\>.\>,\>.\>)$ on a $\g$-module $V$ is called contravariant if
$\bigl(x v,w\bigr)=\bigl(v,\omega(x)w\bigr)$ for all $x\in U_q(\g)$, $v,w\in V$.
A module of highest weight has a unique $\C$-valued contravariant form such that squared norm of the highest vector is $1$.
We call this form canonical and extend this term to a form on tensor products that is the product of canonical forms on tensor factors.
Such a form is contravariant because $\omega$ is a coalgebra map.

Let us recall the definition of $U_q(\h)$-valued Shapovalov form on the Borel subalgebra  $U_q(\b_-)$
that was  introduced for  $U(\g)$ and studied  in  \cite{Shap}.
Regard $U_q(\b_-)$ as a free right $U_q(\h)$-module generated by $U_q(\g_-)$.
The triangular decomposition $U_q(\g)=U_q(\g_-)U_q(\h)U_q(\g_+)$ facilitates projection
   $\wp\colon U_q(\g)\to U_q(\h)$ along the sum $\g_-U_q(\g)+U_q(\g)\g_+$, where $\g_-U_q(\g)$ and $U_q(\g)\g_+$ are
   right and left ideals generated by  positive and negative root vectors,
   respectively.
Set
$$
(x,y)=\wp\bigl(\omega(x)y\bigr), \quad x,y\in U_q(\g).
$$
Thus defined the form is $U_q(\h)$-linear and contravariant.
It follows that the left ideal $U_q(\g)\g_+$ is in the kernel, so the form descends to a
$U_q(\h)$-linear form
on the quotient  $U_q(\g)/U_q(\g)\g_+\simeq U_q(\b_-)$.

A Verma module $V_\la=U_q(\g)\tp_{U_q(\b_+)}\C_\la$ of highest weight $\la\in \h^*$ is induced from  the 1-dimensional $U_q(\b_+)$-module $\C_\la$ that is trivial on $U_q(\g_+)$ and
returns  $q^{(\la,\al)}$ on $q^{h_\al}\in U_q(\h)$, $\al \in \Pi$. Its highest vector    is denoted by  $v_\la$, which
is also called vacuum vector. It freely generates $V_\la$  over $U_q(\g_-)$.

Specialization of the Shapovalov form at $\la\in \h^*$ yields the canonical contravariant $\C$-valued form $(x,y)_\la=\la\Bigl(\wp\bigl(\omega(x)y\bigr)\Bigr)$
on $V_\la$, upon a natural  $U_q(\g_-)$-module isomorphism $U_q(\g_-)\simeq V_\la$
extending the assignment $1\mapsto v_\la$. Conversely, the canonical contravariant form
on $V_\la$ regarded as a function of $\la$ descends to the  Shapovalov form if one views $U_q(\h)$ as
the algebra of polynomial functions on $\h^*$.
By  an abuse of terminology, we also mean by Shapovalov form the canonical
contravariant form on $V_\la$.

It is known from \cite{DCK} that the contravariant form on $V_\la$  module goes degenerate if and only if
its highest weight is in the union of
\be
\Hc_{\bt,m}=\{\la\in \h^*\>|\>q^{2(\la+\rho,\bt)-m(\bt,\bt)}=1\}
\label{Kac-Kazhdan}
\ee over $\bt \in \Rm^+$ and $m\in \N$, where
$\rho=\frac{1}{2}\sum_{\al\in \Rm^+}\al$.
In the classical case $q=1$, $\Hc_{\bt,m}$ becomes  a Kac-Kazhdan hyperplane of weights satisfying $2(\la+\rho,\bt)=m(\bt,\bt)$.

Recall that a vector $v\in V_\la$ of weight $\la-\mu$ with $\mu\in \Gamma_+$, $\mu\not =0$, is called extremal if $e_\al v=0$ for all $\al \in \Pi$.
 Extremal vectors   are
in the kernel of the contravariant form and generate submodules of the corresponding highest weights.
We will be interested in the special case when $\mu=m\bt$ with $\bt\in \Rm^+$ and $m\in \N$.
Then the
highest weight $\la$ has to be in $\Hc_{\bt,m}$.
The image $\theta_{\bt,m}$ of $v$ under the isomorphism $V_\la\to U_q(\g_-)$ is called Shapovalov element of a positive root $\bt$ and degree $m$.

For simple $\bt$ the  element $\theta_{\bt,m}$  is just the $m$-th power of
the root vector, $\theta_{\bt,m}=f_\bt^m$. For non-simple $\bt$, it is a rational trigonometric function $\Hc_{\bt,m}\to U_q(\g_-)$.
The goal of this work is to find explicit expressions for $\theta_{\bt,m}$ with non-simple $\bt$.

\section{Shapovalov inverse form and its matrix elements}
Define an opposite Verma $U_q(\g)$-module $V_\la'$ of lowest weight $-\la$ as follows. The underlying vector
space of $V_\la'$ is taken to be $V_\la$, while the representation homomorphism
$\pi'_\la$ is twisted by $\si$, that is $\pi'_\la=\pi_\la\circ \si$.
The module $V_\la'$ is freely generated over $U_q(\g_+)$ by its lowest vector $v_\la'$.

Let $\si_\la\colon V_\la\to V_\la'$ denote the isomorphism of vector spaces,
$
x v_\la\mapsto \si(x) v_\la', \quad x\in U_q(\g_-).
$
It intertwines the representations homomorphisms $\pi_\la' \circ \si= \si_\la\circ \pi_\la$.
This map  relates the contravariant form on $V_\la$ with  a $U_q(\g)$-invariant pairing
$V_\la\tp V_\la'\to V_\la\tp V_\la \to \C$.

Suppose that the module $V_\la$ is irreducible. Then its invariant pairing is non-degenerate (as well as the contravariant form on $V_\la$).
The inverse form belongs to a completed tensor product $V_\la'\hat \tp V_\la$.
Under the isomorphisms $V_\la\to U_q(\g_-)$, $V_\la'\to U_q(\g_+)$, it goes to an element that we denote by $\Sc\in  U_q(\g_+)\hat \tp  U_q(\g_-)$
and call universal Shapovalov matrix. Given a $U_q(\g_+)$-locally nilpotent  $U_q(\g)$-module $V$ with representation homomorphism
$\pi\colon U_q(\g)\to \End(V)$ the image $S=(\pi\tp \id)(\Sc)$ is a matrix with entries in $U_q(\g_-)$.
It features a rational trigonometric (rational in the classical case) dependance  on  $\la\in \h^*$.
We will assume that  $V$ is diagonalizable with finite dimensional weight spaces. We will also assume that $V$ is endowed with a non-degenerate contravariant form,
for instance, if $V$ is a tensor power  of an irreducible module of highest weight. Using terminology
adopted in the quantum inverse scattering theory, we call the module $V$ auxiliary.

Varying the highest weight $\la$ we get a rational trigonometric dependance of $\Sc$. As a function
of $\la$,  $\Sc$ is  regarded as an element of $U_q(\g_+)\hat \tp \hat U_q(\b_-)$, where $\hat U_q(\b_-)$ is viewed as a right $\hat U_q(\h)$-module
freely generated by $U_q(\g_-)$. This way
the weight dependance is accommodated by the right tensor leg of $\Sc$.

An explicit expression of $S$ in a weight basis  $\{v_i\}_{i\in I}\subset V$, $v_i\in V[\nu_i]$, can be formulated in terms of  Hasse diagram,
$\mathfrak{H}(V)$. Such a diagram is
associated with any partially ordered sets. In our case the partial ordering is induced by the $U_q(\g_+)$-action on $V$.
Nodes are elements of the  basis $\{v_i\}_{i\in I}$.
Arrows are simple root vectors $e_\al$ connecting the nodes $v_i\stackrel{e_\al}{\longleftarrow} v_j$
whose weight difference is $\nu_i-\nu_j=\al$. Then a node $v_i$ is succeeding a node $v_j$ if
 $\nu_i -\nu_j\in \Gamma_+\backslash\{0\}$.
The matrix $S$ is triangular: $s_{ii}=1$ and $ s _{ij}=0$ if $\nu_i\not \succ\nu_j$. The entry $s_{ij}$
is a rational trigonometric function $\h^*\to U_q(\g_-)$ taking values in the subspace of  weight $\nu_j-\nu_i\in -\Gamma_+$.
It is also convenient to introduce a stronger partial ordering as we will explain below.

Clearly the matrix $S$ depends only on the $U_q(\b_+)$-module structure on $V$. In order to calculate a particular
element $s_{ij}$, we can choose a weight basis that extends a  basis in the cyclic submodule $U_q(\g_+)v_j$.
Then, in particular, $s_{ij}=0$ if  $v_i\not \in U_q(\g_+)v_j$.

We define a Hasse sub-diagram
$\mathfrak{H}(v_i,v_j)\subset\mathfrak{H}(V)$ that comprises  all possible routes from $v_j$ to $v_i$. A node $v_k\in\mathfrak{H}(V)$
is in $\mathfrak{H}(v_i,v_j)$ if and only if $v_i\succeq v_k\succeq v_j$.
The sub-diagram $\mathfrak{H}(v_i,v_j)$ is associated with a $U_q(\g_+)$-module $V(v_i,v_j)$ that is the quotient of $U_q(\g_+)v_j$ by the sum of cyclic submodules
 $U_q(\g_+)v_k\subset U_q(\g_+)v_j$ where $v_k\not \in  \mathfrak{H}(v_i,v_j)$. It is the module $V(v_i,v_j)$ that is needed
to calculate a matrix element $s_{ij}$.

We  recall a construction of $S$  following \cite{M1}.
Let $\{h_i\}_{i=1}^{\rk \g}\in \h$ be an orthonormal basis. The element $q^{\sum_i h_i\tp h_i}$ belongs to a completion of $U_q(\h)\tp U_q(\h)$
in the $\hbar =\ln q$-adic topology. Choose an R-matrix $\Ru$ of $U_q(\g)$ such that $\check \Ru=q^{-\sum_i h_i\tp h_i}\Ru\in U_q(\g_+)\hat \tp U_q(\g_-)$
and set $\Cc=\frac{1}{q-q^{-1}}(\check \Ru- 1\tp 1)$.
The key identity on $\Cc$ that facilitates the $q$-version of the theory is  \cite{M1}
\be
[1\tp e_{\al},\Cc] +(e_{\al}\tp q^{-h_{\al}})\Cc- \Cc (e_{\al}\tp q^{h_{\al}}) =
e_{\al}\tp [h_{\al}]_q, \quad \forall \al \in \Pi.
\label{key_id}
\ee
In the classical limit, $\Cc=\sum_{\al\in  \Rm^+} e_\al\tp f_\al$ becomes  the polarized split Casimir of $\g$ without its Cartan part.
One then recovers an identity
\be
[1\tp e_{\al},\Cc]+[e_{\al}\tp 1,\Cc] = e_{\al}\tp h_{\al}
 \label{key_id_cl}
\ee
for each simple root $\al$.

Let $c_{ij}\in U_q(\g_-)$ denote the entries of the matrix $(\pi\tp \id)(\Cc)\in \End(V)\tp U_q(\g_-)$.
We rectify the partial ordering and the Hasse diagram $\mathfrak{H}(V)$ by removing arrows $v_i\leftarrow v_j$ if $c_{ij}=0$.
This will not affect the formula (\ref{norm_sing_vec}) for matrix elements of $S$.

For each weight $\mu\in \Gamma_+$ put
\be
\eta_{\mu}=h_\mu+(\mu, \rho)-\frac{1}{2}(\mu,\mu) \in \h\op \C.
\ee
Regard $\eta_{\mu}$  as an affine function on $\h^*$ by  the assignment $\eta_\mu\colon \zt \mapsto (\mu,\zt+ \rho)-\frac{1}{2}(\mu,\mu)$, $\zt\in \h^*$.
Observe that $\eta_{m\bt}=m\bigl(h_\bt+(\bt, \rho)-\frac{m}{2}(\bt,\bt)\bigr)$. That is,  $[\eta_{m\bt}(\la)]_q$
vanishes on $\Hc_{\bt,m}$ (and only on $\Hc_{\bt,m}$ in the classical case).

For a pair of non-zero vectors $v,w\in V$ define a matrix element $\langle w|v\rangle=(w,\Sc_1v)\Sc_2\in \hat U_q(\b_-)$,
where  $\Sc_1\tp \Sc_2$ stands for  a Sweedler-like notation for   $\Sc$ and the pairing is with respect to a non-degenerate
contravariant form on $V$.
Its specialization at a weight $\la$ is denoted by $\langle w|v\rangle_\la$, which can be determined from the equality
$\langle w|v\rangle_\la v_\la= \langle w|v\rangle v_\la\in V_\la$.
For each $w$ from $V$, the map $V\to V_\la$, $v\mapsto \langle v|w\rangle v_\la$ satisfies:
$e_\al\langle  v|w\rangle v_\la=\langle \si(e_\al) v|w\rangle v_\la$ for all $\al\in \Pi$.
This is a consequence of  $\U_q(\g_+)$-invariance of the tensor $\S(1\tp v_\la)\in U_q(\g_+)\hat \tp V_\la$.

The matrix element $\langle v_i|v_j\rangle$ equals  $s_{ij}$  if $(v_i,v_k)=\dt_{ik}$ for all  $k\in I$.
It will be always the case in what follows.

Fix a "start" node $v_a$ and an "end" node $v_b$ such that $v_b\succ v_a$.
Then a re-scaled matrix element   $\check s_{ab}=-s_{ba}[\eta_{\nu_b-\nu_a}]_qq^{-\eta_{\nu_b-\nu_a}}$
can be calculated by the formula
\be
\check s_{ba}=c_{ba}+\sum_{k\geqslant 1}\sum_{v_b \succ v_k\succ \ldots \succ v_1\succ v_a}
c_{b k}\ldots
c_{1 a}\frac{(-1)^kq^{\eta_{\mu_k}}\ldots q^{\eta_{\mu_1}}}{[\eta_{\mu_k}]_q\ldots [\eta_{\mu_1}]_q} \in \hat U_q(\b_-),
\label{norm_sing_vec}
\ee
where
$
 \mu_l=\nu_{l}-\nu_{a}\in \Gamma_+$, $l=1,\ldots, k.
$
Here the summation is performed over all possible routes (sequences of ordered nodes) from $v_a$ to $v_b$, see \cite{M1} for details.

It is straightforward that $U_q(\g_+)$-invariance of the tensor $\Sc(v_a\tp v_\la)$ implies
\be
e_\al \check s_{ba}(\la)v_\la \propto [\eta_{\nu_b-\nu_a}(\la)]_q \sum_{k}\pi(e_\al)_{bk} s_{ka}(\la)v_\la.
\label{e-al-s_ij}
\ee
The matrix entries $s_{ka}(\la)$ carry weight $-(\nu_b-\nu_a-\al)$.
It follows that $\check s_{ba}(\la)v_\la$ is an extremal vector in $V_\la$ for $\la$ satisfying  $[\eta_{\nu_b-\nu_a}(\la)]_q=0$ provided
\begin{enumerate}
  \item $\check s_{ba}(\la)\not =0$,
  \item $\la$ is a regular point for all $s_{ka}(\la)$ and all  $\al$.
\end{enumerate}
We aim to find an appropriate matrix element for $\theta_{\bt,m}$ that satisfies these conditions.

    Let $V$ be a  $U_q(\g)$-module with a pair of vectors  $v_a,v_b\in V$ such that $e_\bt v_a=v_b$
    for $\bt \in \Rm^+$.
    We call the triple $(V,v_b,v_a)$ a  $\bt$-representation.

\begin{propn}
\label{Shapdeg1}
  Let $(V,v_b,v_a)$ be a $\bt$-representation for  $\bt \in \Rm^+$. Then for generic $\la\in \Hc_{\bt,1}$
   the vector $\check s_{ba}(\la)v_\la\in V_\la$ is extremal.
\end{propn}
\begin{proof}
  The factors $\frac{q^{\eta_{\mu_k}}}{[\eta_{\mu_k}]_q}$  in (\ref{norm_sing_vec})  go singular on the union of a finite number of
  the null-sets $\{\la\in \h^*\>|\>[\eta_{\mu_k}(\la)]_q=0\}$. None of   $\mu_k$ is collinear to $\bt$, hence $\check s_{ba}(\la)$ is regular at generic $\la\in \Hc_{\bt,1}$.
  By the same reasoning, all $s_{ka}(\la)$ in (\ref{e-al-s_ij}) are regular at such $\la$. Finally,
  the first term  $c_{ba}$ (and only this one) involves the Lusztig root vector $f_\bt$, a generator of a PBW basis
  in $U_q(\g_-)$. It is therefore independent of the
  other terms, and  $\check s_{ba}(\la)\not=0$.
\end{proof}

Upon identification of $\hat U_q(\b_-)$ with  rational $U_q(\g_-)$-valued functions on $\h^*$  we
conclude that $\check s_{ba}$ is a Shapovalov element  $\theta_{\bt,1}$ and denote it by  $\theta_\bt$.
Uniqueness of extremal vector of given weight implies that all matrix elements $\check s_{ba}$ with $v_b=e_\bt v_a$ deliver the same $\theta_\bt$, up to a scalar factor. However, they are generally different at $\la\not \in \Hc_{\bt,m}$. When we aim at $\theta_{\bt,m}$ with $m>1$, we have to choose matrix elements for $\theta_\bt$ more carefully in order to use them as building blocks.

Note that it was relatively easy to secure the above two conditions in the case of $m=1$. For higher $m$ we will opt a different strategy:
we will satisfy the first condition by the very construction and bypass a proof of the second with different arguments.

\section{Factorization of Shapovalov elements}

For a positive root  $\bt\in \Pi$ denote by  $\Pi_\bt\subset \Pi$ the set of simple roots entering
the expansion of $\bt$ over the basis $\Pi$ with  positive coefficients.
A simple Lie subalgebra, $\g(\bt)\subset \g$, generated by
$e_\al,f_\al$ with $\al\in \Pi_\bt$ is called  support of $\bt$. Its universal enveloping algebra is quantized as a Hopf subalgebra in $U_q(\g)$.
\begin{definition}
  Let  $\bt \in \Rm^+$ be a positive root and $(V,v_b,v_a)$ a $\bt$-representation such that  $e_\al v_b=0$ for all $\al \in \Pi$,
   and $(\nu_b,\bt^\vee)=1$. We call such $\bt$-representation admissible.
\end{definition}

If a triple is $(V,v_b,v_a)$ is  admissible then $v_b$ is the highest vector of a $U_q(\g)$-submodule in $V$.
For finite dimensional  $\dim V<\infty$, $v_b$ generates a 2-dimensional submodule of the $\s\l(2)$-subalgebra generated by $f_\bt,e_\bt$.
The vector $v_b$   can be included in an orthonormal basis in $V$, as required.

 \begin{lemma}
\label{matr_factorization}
Let $(V,v_b,v_a)$ be an admissible $\bt$-representation.
Set $v_b^m=   v_b^{\tp m}\in V^{\tp m}$ for $m\in \N$.
Pick up $\la\in \h^*$ such that all Verma modules $V_{\la_k}$  with $\la_k=\la+k\nu_a$, $k =0, \ldots, m-1$, are irreducible.
Then  there is $v_a^m \in V^{\tp m}$ of weight $m \nu_a$ such that
  \be
\langle v_b^m| v_a^m\rangle_{\la_0}  = \langle v_b|v_a\rangle_{\la_{m-1}}\ldots \langle v_b|v_a\rangle_{\la_0}.
\label{factor-root-degree}
\ee
\end{lemma}
\begin{proof}
Let $\la$ satisfy the required conditions.
There is an equivariant map $\varphi_k\colon V_{\la_k}\to V\tp V_{\la_{k-1}}$ sending the
highest vector $v_{\la_k}$ to an extremal vector $\Sc(v_a\tp v_{\la_{k-1}})\in V\tp V_{\la_{k-1}}$.
Here $\Sc$ is the universal Shapovalov matrix  of $V_{\la_{k-1}}$.
Consider a chain of module homomorphisms
$$
V_{\la_{m}}\stackrel{\varphi_m}{\longrightarrow} V\tp V_{\la_{m-1}}\stackrel{\id_1\tp \varphi_{m-1}}{\longrightarrow} V\tp (V\tp V_{\la_{m-2}})\to \ldots
\stackrel{\id_{m-1}\tp \varphi_1}{\longrightarrow}  V^{\tp (m-1)}\tp (V\tp V_{\la_0}),
$$
where  $\id_k$ are the identity operators on $V^{\tp k}$.
The vector $v_{\la_m}$ eventually goes over to   $\Sc(\tilde v^m_a\tp v_{\la_0})$, where $\tilde v^m_a\in V^{\tp m}$ is  of weight $m\nu_a$.
It is related with $v_a^{\tp m}$
by an invertible operator  from $\End(V^{\tp m})$, which is  $m-1$-fold dynamical twist \cite{ES}.

Let us  calculate  $\langle v_b^{ m}|\tilde v^m_a\rangle_{\la_0}$ by
pairing the tensor leg of $\Sc(\tilde v^m_a\tp v_{\la_0})$ with $v_b^{m}=v_b\tp v_b^{m-1}$.
Using  equality   $\Sc(\tilde v^m_a\tp v_{\la_0})=\Sc\bigl( v_a\tp  \Sc(\tilde v^{m-1}_a\tp v_{\la_0})\bigr)$
we reduce $\langle v_b^{ m}|\tilde v^m_a\rangle_{\la_0}$ to
$$
\bigl(v_b^{m-1},    \langle v_b| v_a \rangle_{\la_{m-1}}^{(1)}\Sc_{1} \tilde v^{m-1}_a \bigr) \> \langle v_b| v_a \rangle_{\la_{m-1}}^{(2)}\> \Sc_{2}(\la_{0})
=
\langle v_b| v_a \rangle_{\la_{m-1}}^{(2)}\Bigl\langle \omega\bigl(\langle v_b| v_a \rangle_{\la_{m-1}}^{(1)}\bigr)v_b^{m-1}|\tilde v^{m-1}_a\Bigr\rangle_{\la_0},
$$
where we use the Sweedler notation $\Delta(x)=x^{(1)}\tp x^{(2)}\in U_q(\b_-)\tp U_q(\g_-)$ for the coproduct of $x\in U_q(\g_-)$.
Since $yq^{h_\al} v_b=\eps(y)q^{(\al,\bt)}  v_b$ for all $y \in U_q(\g_+)$ and $\al \in \Ga_+$,  we
arrive at
$$
\langle v_b^{m}|\tilde v^m_a\rangle_{\la_0}=q^{-(\bt,\nu_b)}\langle v_b |v_a\rangle_{\la_{m-1}}\langle v_b^{m-1}|\tilde v^{m-1}_a\rangle_{\la_0}.
$$
Proceeding by induction on $m$ we conclude that $\langle v_b^{ m}|\tilde v^m_a\rangle_{\la_0}$ equals the right-hand side of (\ref{factor-root-degree}),
up to the factor $q^{-m(\bt,\nu_b)}$. Finally, set $v^m_a=q^{m(\bt,\nu_b)} \tilde v^m_a$.
 This proves the lemma for generic and hence for all $\la$ where the right-hand side of (\ref{factor-root-degree}) makes sense.
\end{proof}
It follows from the above factorization that the least common denominator of the    extremal vector
$
u=\Sc (v^m_a \tp v_{\la})\in V^{\tp m}\tp V_{\la}
$
contains
$$
d(\la)=[\eta_\bt(\la+(m-1)\nu_a)]_q=[(\la+\rho,\bt)-\frac{m}{2}(\bt,\bt)]_q.
$$
It comes from the leftmost factor $\langle v_b|v_a\rangle_{\la_{m-1}}$
in the right-hand side of (\ref{factor-root-degree}).
Denote by  $s_{v_b^m,v_a^m}(\la)$ the matrix element $\langle v_b^m| v_a^m\rangle_{\la}$.
Since $d$ divides $[\eta_{m\bt}]_q$, the re-scaled matrix element
$$
\check s_{v_b^m,v_a^m}(\la)=c(\la)d(\la)s_{v_b^m,v_a^m}(\la)\propto  \prod_{k=0}^{m-1}\theta(\la_k),
$$
where $c(\la)=-q^{-\eta_{m\bt}(\la_{m-1})}\frac{[\eta_{m\bt}(\la)]_q}{d(\la)}$, is regular and does not vanish at generic $\la \in \Hc_{\bt,m}$
because $d(\la)$ cancels the pole in $\langle v_b|v_a\rangle_{\la_{m-1}}$.
Put  $\check u =d^k(\la) u$, where  $k\geqslant 1$ is the maximal degree of this  pole in $u$. It
is an extremal vector in $V^{\tp m}\tp V_\la$ that is regular at generic $\la\in \Hc_{\bt,m}$.

Indeed, let $\Hc_\mu$ denote the null set $\{\la\in \h^*|[\eta_\mu(\la)]_q=0\}$ for $\mu \in \Gamma_+$.
Then  the $V_\la$-components of $\check u$ may have poles only at $\la\in  \cup_{\mu<\bt}\Hc_\mu$.
But each $\mu$ is either not collinear to $\bt$
or $\mu=l\bt$ with $l<m$. In both cases the complement to $\Hc_{\bt,m}\cap \Hc_\mu$ is dense in $\Hc_{\bt,m}$ because $q$ is not a root of unity.
\begin{propn}
\label{Prop_factoriztion}
  For generic $\la\in \Hc_{\bt,m}$,
    $\theta_{\bt,m}(\la) \propto \check s_{v_b^m,v_a^m}(\la)$.
\end{propn}
\begin{proof}
The singular vector $\check u$ is presentable as
$$
\check u= v^m_a \tp d^k(\la) v_{\la}+\ldots + v^m_b\tp d^{k-1}(\la)c(\la) \check s_{v_b^m,v_a^m}(\la)v_{\la}.
$$
We argue that
$\check u=   v^m_b\tp c(\la) \check s_{v_b^m,v_a^m}(\la)v_{\la}$ for generic $\la$ in $\Hc_{\bt,m}$, where $d(\la)=0$.
Indeed, the $V_{\la}$-components of $\check u$ span a $U_q(\g_+)$-submodule in $V_\la$.
A vector of maximal weight in this span is extremal and  distinct from $v_{\la}$. But
$\theta_{\bt,m}(\la)v_{\la}$ is the only, up to a factor,  extremal vector in $V_{\la}$, for generic $\la$.
Therefore $k=1$ and $\theta_{\bt,m}\propto \check s_{v_b^m,v_a^m}$.
\end{proof}
An admissible $\bt$-representation can be associated with every simple root $\al\in \Pi_\bt$ if one
sets
$V$ to be the irreducible module of highest weight $\frac{(\bt,\bt)}{\ell(\al,\al)}\omega_\al$, where $\ell=\ell_{\al,\bt}$ is the
multiplicity of $\al$ with which it enters $\bt$. We denote this module by $V_{\al,\bt}$.
It is finite dimensional if $\frac{(\bt,\bt)}{\ell(\al,\al)}\in \N$.
Otherwise it is a parabolic Verma module relative to a Levi subalgebra   with the root basis $\Pi\backslash \{\al\}$, cf. the next section.

 One can pass to the "universal form" of $\theta_{\bt}$ regarding it as an element of $\hat U_q(\b_-)$.  Then
\be
 \theta_{\bt,m}=  (\tau_{\nu_b}^{m-1}\theta_\bt) \>\ldots \> (\tau_{\nu_b}\theta_\bt)\>\theta_\bt,
\label{factor-root-degre-un}
\ee
where $\tau_\nu$ is an automorphism of $\hat U_q(\h)$ generated by the affine shift of $\h^*$ by the weight $\nu$,
that is,
$(\tau_\nu\varphi)(\mu) = \varphi(\mu+\nu)$, $\varphi\in \hat U_q(\h)$,   $\mu\in \h^*$.
One may ask when the shift is trivial, $\tau_{\nu_b}\theta_{\bt}=\theta_{\bt}$, and   $\theta_{\bt,m}$ is just the $m$-th power of $\theta_\bt$.
\begin{propn}
\label{plain-power}
  Let $\bt$ be a positive root. Suppose that there is $\al\in \Pi_\bt$  with $\ell_{\al,\bt}=1$. Then $\theta_{\bt,m}=\theta_{\bt}^m\in U_q(\b_-)$.
\end{propn}
\begin{proof}
 Let $\s\subset \g$ be a semi-simple subalgebra generated by simple root vectors $f_\mu,e_\mu$ with $\mu\not=\al$.
Take for  $V$ the  module $V_{\al,\bt}$ with highest weight $\phi= \frac{(\bt,\bt)}{(\al,\al)}\omega_\al$.
Put     $v_b$ to be the highest vector and $v_a\propto f_\bt v_b$.

Both $v_a$ and $v_b$ can be included in
an orthonormal basis because they span their weight subspaces in $V$. Therefore $\langle v_b|v_a\rangle= s_{ba}$  can be calculated by formula (\ref{norm_sing_vec}).
We  write it as
$$
\theta_\bt(\la)=c_{ba}+\sum_{v_b\succ v_i\succ v_a} c_{bi}s_{ia}(\la).
$$
The highest vector $v_b$ is killed by $\s_-$, therefore the Hasse diagram between $v_a$ and $v_b$ is
$$
v_b\quad\stackrel{e_{\al}}{\longleftarrow}\quad f_\al v_b\quad \ldots \quad  v_a,
$$
where arrows in  the suppressed part are simple root vectors from  $U_q(\s_+)$.
But then  the only copy of $f_\al$ is in  $c_{bi}$ while all $s_{ia}$ belong to $U_q(\s_-)\hat U_q(\h_\s)$,
the extended Borel subalgebra of $U_q(\s)$.

Finally, since $\Pi_\s$ is orthogonal to $\nu_b$, we have $(\mu, \nu_a)=-(\mu,\bt)$ for all $\mu\in \Rm^+_\s$.
Therefore
$$
\theta_\bt(\la_k)=\theta_\bt(\la-k\bt)
,\quad
\theta_{\bt,m}(\la)=\prod_{k=0}^{m-1}\theta_\bt(\la-k\bt),
$$
where the product is taken in the descending order from left to right.
This proves the plain power factorization because each $\theta_\bt$ carries weight $-\bt$.
\end{proof}
Conditions of the above proposition are fulfilled for all pairs $\al, \bt$ in the case of $\s\l(n)$.
\section{Shapovalov elements of degree 1}
In this section we describe the factor   $\theta_{\bt}$ entering  (\ref{factor-root-degre-un}),
for a particular admissible $\bt$-representation $(V,v_b,v_a)$.
We give a complete  solution to the problem in the classical case.
In the case of  $q\not =1$, we do it up to calculation of the entries of the matrix $\Cc$ in a  simple finite dimensional module $\tilde \g$ that is a $q$-deformation
of the adjoint module $\g$. Its  highest weight is the maximal root, $\xi\in \Rm^+$.

To achieve our goals, we need to figure out the Hasse sub-diagram
$\mathfrak{H}(v_b,v_a)\subset\mathfrak{H}(V)$ that comprises  all possible routes from $v_a$ to $v_b$.
We argue that $\mathfrak{H}(v_b,v_a)$ can be extracted from a diagram $\mathfrak{H}( \b_-)$ which we introduce below, and the underlying  $U_q(\g_+)$-modules
are isomorphic.

The  $U_q(\g_+)$-module  associated with $\mathfrak{H}( \b_-)$  is constructed from
$\tilde \g$ by  factoring  out the span of positive weight spaces.
In order to distinguish the case of $q\not =1$ from classical and to avoid confusion with root vectors,
we will mark the nodes with tilde.
Vectors  $\tilde f_\eta$  of weights $-\eta\in - \Rm^+$  are defined uniquely up to a sign if we normalize them by $(\tilde f_\eta,\tilde f_\eta)=1$.
We may assume that they are deformations of classical root vectors.
 We take   $\tilde h_\al= e_\al \tilde f_\al$, $\al\in \Pi$, for  basis elements of zero weight.

For example, the diagram $\mathfrak{H}(\b_-)$ in the case  of  $\g=\g_2$ is
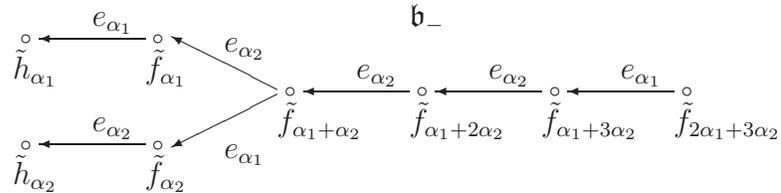
\begin{center}
\begin{picture}(310,70)
\put(155,60){$\b_-$}

\put(35,20){$e_{\al_2}$}
\put(85,10){$e_{\al_1}$}
\put(135,40){$e_{\al_2}$}
\put(185,40){$e_{\al_2}$}
\put(235,40){$e_{\al_1}$}

\put(5,0){$\tilde h_{\al_2}$}
\put(55,0){$\tilde f_{\al_2}$}
\put(105,20){$\tilde f_{\al_1+\al_2}$}
\put(155,20){$\tilde f_{\al_1+2\al_2} $}
\put(205,20){$\tilde f_{\al_1+3\al_2}$}
\put(255,20){$\tilde f_{2\al_1+3\al_2}$}

\put(10,15){\circle{3}}
\put(60,15){\circle{3}}
\put(110,35){\circle{3}}
\put(160,35){\circle{3}}
\put(210,35){\circle{3}}
\put(260,35){\circle{3}}

\put(55,15){\vector(-1,0){40}}
\put(105,34){\vector(-2,-1){40}}
\put(155,35){\vector(-1,0){40}}
\put(205,35){\vector(-1,0){40}}
\put(255,35){\vector(-1,0){40}}

\put(35,60){$e_{\al_1}$}
\put(85,50){$e_{\al_2}$}

\put(5,40){$\tilde h_{\al_1}$}
\put(55,40){$\tilde f_{\al_1}$}

\put(10,55){\circle{3}}
\put(60,55){\circle{3}}

\put(55,55){\vector(-1,0){40}}
\put(105,36){\vector(-2,1){40}}


\end{picture}
\end{center}

From now on we fix  $V=V_{\al,\bt}$ with highest weight  $\phi=\frac{(\bt,\bt)}{\ell_{\al,\bt}(\al,\al)}\omega_\al$   and  highest vector $v_b$.
We denote by  $\l\subset \g$  a reductive Lie subalgebra of maximal rank whose root system is $\Pi_\l=\Rm\backslash\{\al\}$
and by  $\p=\l+\g_+$  its  parabolic extension.


In order to construct the start node $v_a\in V$, we will use the following observation.
Recall that a singular vector $\sum_{i}w_i\tp v_i$ in a tensor product  $W\tp V$ of two irreducible modules of highest weight defines
a $U_q(\g_+)$-homomorphism $W^*\to V$ (and respectively $V^*\to W$). Here $W^*$ is an irreducible $U_q(\g)$-module of lowest weight,
which is negative the highest weight of $W$.  The dual action is defined  with the help of antipode  $\gm$ in the standard way:
$(x\varphi)(w)=\varphi\bigl(\gm(x)w\bigr)$, for  $x\in U_q(\g_+)$,  $w\in W$, and $\varphi\in W^*$.
The homomorphism $W^*\to V$ is implemented via the assignment
$\varphi\mapsto \sum_{i}\varphi(w_i) v_i$. We will apply this construction  to $W=\tilde \g$.


\begin{lemma}
  There exists a unique, up to a scalar factor, singular vector $u\in \tilde \g\tp V $ of weight $\phi$.
 \label{sing-vec}
 \end{lemma}
\begin{proof}
Let $J\subset U_q(\g_-)$ be the annihilator of the highest vector $v_b\in V$.
Singular vectors in $\tilde \g\tp V$ of weight $\phi $ are in bijection with vectors $\tilde h\in \tilde \g$ of zero weight
killed by the left ideal $\si(J)\subset U_q(\g_+)$.
Pick up $\tilde h\not =0$ orthogonal to all $\mu\in \Pi_\l$; it is unique up to a scalar factor.

The ideal $J$ is generated by elements  $\theta\in U_q(\g_+)$ such that $\theta v_b$ are singular vectors in the Verma module $V_\phi$
covering $V$. By construction, $\tilde u$ is killed by $e_\al\in J$ with $\al\in \Pi_\l$. If   $\theta v_\phi\in V_\phi$  is a singular vector of weight $\phi-m\eta$ with  $\eta \in \Rm^+\backslash \Rm^+_\l$,
then $m>1$.
Indeed, since   $\phi=l\omega_\al$ with  positive rational $l=\frac{(\bt,\bt)}{\ell_{\al,\bt}(\al,\al)}$,
we have an inequality    $l(\omega_\al, \eta^\vee)+(\rho, \eta^\vee)>1$. Then  the condition (\ref{Kac-Kazhdan}), where
$\la$ is replaced with $\phi$ and $\bt$ with $\eta$, is fulfilled only if $m>1$,
since   $q$ is not a root of unity.
Then the element $\si(\theta )$ kills $\tilde h$ because $m\eta $  with $m>1$ is not a weight of $\tilde \g$.
\end{proof}
Remark that   $V$ is finite dimensional if $\frac{(\bt,\bt)}{\ell_{\al,\bt}(\al,\al)}\in \Z$ and a parabolic Verma module
otherwise because its highest weight is away from De Concini-Kac-Kazhdan hyperplanes $\Hc_{\eta,m}$ with $\eta\in \Rm^+\backslash \Rm^+_\l$.

Now let $v_a\in V$ be the vector of minimal weight in the expansion
 $u=\tilde e_\xi\tp v_a+\ldots $ over the chosen  basis in $\tilde \g$ (we have omitted the terms of lower  weights in the $\tilde \g$-factor).
Notice that in the classical case the vector $f_\eta v_b$ does not vanish if $\eta\in \Rm^+\backslash \Rm^+_\l$ because
$(\eta,\phi)>0$. In particular,  $v_a\propto f_\xi v_b\not =0$ for the maximal root $\xi$.
For general $q$, $ v_a$ is killed by the left ideal in $U_q(\g_+)$ annihilating the lowest vector $\tilde f_\xi\in \tilde \g\simeq \tilde \g^*$,
Such $v_a$  is unique in $V$ up to a scalar factor, because of Lemma \ref{sing-vec}.


Introduce  a partial order on positive roots by writing $\mu\prec \nu$ {\em iff} $f_\mu\succ f_\nu$ in $\Hg(\b_-)$. This is in agreement
with the partial order on  $\Hg(\g_+)\subset \Hg(\g)$, which   is exactly the  Hasse diagram of the root system $\Rm^+$, \cite{Pan}. Note that  $\al\prec \bt$ for simple $\al$ if and only if $\al\in \Pi_\bt$.
\begin{propn}
\label{Hasse_aux_mod}
Let $u=\tilde e_\xi\tp v_a+\ldots $ be the singular vector from Lemma \ref{sing-vec} with  $v_a\in V$ of minimal weight in the expansion
over a weight basis in $\tilde \g$.
  Then the $U_q(\g_+)$-module generated by $v_a\in V$ is isomorphic to
  $\tilde \g(\tilde h_\al, \tilde f_\xi)$, for almost all $q$.
\end{propn}
\begin{proof}
The $U_q(\g_+)$-module homomorphism $\tilde \g\to V$ determined by the assignment $\tilde f_\xi\mapsto  v_a$  factors through the quotient
 $\g(\tilde h_\al,\tilde f_\xi)$
because the kernel includes all $\tilde f_\eta$ with $\eta \in \Rm^+_\l$,
all $\tilde h_\eta=e_\mu \tilde f_\eta$ with $\eta \in \Pi_\l$, and all negative weight spaces.
We are left to prove that it is  an isomorphism on $\g(\tilde h_\al,\tilde f_\xi)$ for almost all $q$.
It is sufficient to check that it is injective for $q=1$ because $V$ rationally depends on $q$.
  But then for each positive root $\eta$ subject to  $\al\preceq \eta\preceq \xi$ the vector $f_\eta v_b$ is in $U(\g_+) f_\xi v_b$ and  is not zero,
  because $(\eta,\phi)>0$.
\end{proof}
\noindent
 It follows that   $e_\bt v_a\not =0$ because  $e_\bt \tilde f_\xi \not =0$. Therefore  $(V,v_b,v_a)$ is an admissible $\bt$-representation
for almost all $q$.

Let us consider the classical case in more detail.
We choose   $h^\vee_\al=\frac{2}{(\al,\al)}h_\al$, $\al\in \Pi$, as a basis in $\h\subset \b_-$,
so that $\al(h^\vee_\al)=2$. The   root vectors $f_\mu$ with  $\mu\in \Rm^+$  form a basis in $\g_-$.
Arrows  labeled by $\al\in \Pi$ are
$h^\vee_\al \stackrel{e_\al}\longleftarrow f_\al$ and $f_\mu \stackrel{e_\al}\longleftarrow f_\nu$  if $\mu=\nu-\al$ is a positive root.
The $U(\g_+)$-module underlying  $\Hg(\b_-)$ is $\g/\g_+$.

Specialization of the formula (\ref{norm_sing_vec}) for $\theta_{\bt}$
requires the knowledge of   matrix  $C=(\pi\tp \id)(\Cc)\in \End(V)\tp U_q(\g_-)$, which is readily available for $q=1$.
For $\nu,\gm\in \Rm^+$, denote by  $C_{\nu,\gm}\in \C$ the scalars such that
$[e_\nu,f_\gm]=C_{\nu, \gm}f_{\gm-\nu}$,
if $\gm-\nu\in \Rm^+$,
$C_{\gm,\gm}=\frac{(\bt,\bt)}{2}\frac{\ell_{\al,\gm}}{\ell_{\al,\bt}}$, and
$C_{\nu,\gm}=0$ otherwise.
Then
$$
(\pi\tp \id)(\Cc)(f_\gm v_b\tp 1)=v_b\tp C_{\gm,\gm}f_\gm+\sum_{\nu\prec \gm} f_{\gm-\nu}v_b\tp C_{\nu,\gm} f_\nu,
$$
for all $\gm$ satisfying $\al\preceq \gm\preceq \bt$. This equality yields all entries of the matrix $C$ needed.
The formula (\ref{norm_sing_vec}) becomes
\be
\theta_{\bt}=
C_{\bt,\bt} f_{\bt}+\sum_{k\geqslant 1}\sum_{\nu_1+\ldots+\nu_{k+1}=\bt}
(C_{\nu_{k+1},\gm_{k}}\ldots C_{\nu_{1},\gm_0})(f_{\nu_{k+1}}\ldots
f_{\nu_1})\frac{(-1)^k }{\eta_{\mu_k}\ldots \eta_{\mu_1}}.
\label{norm_sing_vec-1}
\ee
The internal summation is performed over all partitions
of  $\bt$ to a sum of  $\nu_i\in \Rm^+$ such that
all   $\gm_i=\gm_{i-1}-\nu_i$ for $i=1,\ldots, k$ with $\gm_0=\bt$ are in $\Rm^+$ and
subject to $\al \preceq \gm_i $. In particular, $\gm_k=\nu_{k+1}$.
The weights $\mu_i$ are defined to be $\mu_i=\gm_0-\gm_i=\nu_1+\ldots+\nu_i$. Note that in the $q\not =1$ case
the corresponding sum  may involve terms with entries of $C$
whose weights are not roots.

Now  we  summarise the results of this paper.
\begin{thm}
\label{classifying Hasse}
  For each $\al \prec \bt$, the rescaled  matrix element $\langle \tilde h_\al|\tilde f_\bt\rangle[\eta_\bt]_q$ with  $\tilde h_\al,\tilde f_\bt\in \tilde \g$,
  is  a Shapovalov element
  $\theta_{\bt,1}$.
  For general degree $m>1$, $\theta_{\bt,m}$ is given by the factorization formula (\ref{factor-root-degre-un}) with
  $\theta_\bt=\theta_{\bt,1}$ and the shift weight $\nu_b=\frac{(\bt,\bt)}{\ell_{\al,\bt}(\al,\al)}\omega_\al$.
\end{thm}
\begin{proof}
  Observe that summation formula (\ref{norm_sing_vec}) involves only   the structure of  $U_q(\g_+)$-module determined by the initial
  and final nodes.
  That is straightforward with regard to the matrix elements of $C$ and also true for the Cartan factors, which  depend only on weight differences
  (mind that weights in a cyclic $U_q(\g_+)$-module  generated by a weight vector are fixed up to a constant weight summand).
  Furthermore, the nodes of the sub-diagram $\Hg(v_a,v_b)$ can be included in an orthonormal basis
  whence  $s_{ba}\propto \langle v_b|v_a\rangle$.
  Now, for almost all $q$, the theorem follows from Proposition \ref{Hasse_aux_mod} and Proposition \ref{Prop_factoriztion} with Lemma \ref{matr_factorization}.
  Therefore it is true for all $q$ where the factors (\ref{factor-root-degre-un}) are defined.
\end{proof}
We remark in conclusion that for fixed $\bt\in \Rm^+$ one can pick up $\al\in \Pi_\bt$ delivering the simplest Hasse diagram $\mathfrak{H}(\tilde h_\al,\tilde f_\bt)$,
e.g. with the smallest fundamental group.
Such diagrams can be found amongst subdiagrams in fundamental auxiliary modules of minimal dimension.
That also applies to their associated $U_q(\g_+)$-modules. For all non-exceptional types of $\g$, the entries of the matrix $\Cc$ participating in the route summation
formula are calculated in \cite{AM}, Proposition 2.2. That is also done for $\g_2$ in \cite{BMO}.
This makes the above description of Shapovalov elements for such quantum groups absolutely
explicit. For exceptional $\g$ of rank $>2$, the problem reduces to calculation of  relevant entries of $\Cc$.

In the context of quantization of semi-simple conjugacy classes \cite{M2}, it is crucial to make sure that
 $\theta_{\bt,m}(\la)$ tends to $f_{\bt}^m$ as $q\to 1$. Factorization (\ref{factor-root-degre-un}) together with the
route summation formula for $\theta_{\bt,1}$ gives important information about
possible singularities of  $\theta_{\bt,m}(\la)$ and facilitate the analysis even without
 knowing the matrix elements of $\Cc$.

\vspace{20pt}
\noindent

\underline{\large \bf Acknowledgement}
\vspace{10pt} \\
 This work is partially supported by the Moscow Institute of Physics and Technology
under  the Priority 2030 Strategic Academic Leadership Program
and by Russian Science Foundation  grant 23-21-00282.
The author thanks Vadim Ostapenko and Vladimir Stukopin for stimulating discussions.



 \end{document}